\documentclass[11pt,leqno]{article}

\usepackage{amssymb,amsmath,epsfig,amsthm}

\newcommand{\n}{\noindent}

\newcommand{\vp}{\varepsilon}
\newcommand{\ovl}{\overline}
\newcommand{\bb}[1]{\mathbb{#1}}
\newcommand{\wt}{\widetilde}

\numberwithin{equation}{section}

\theoremstyle{plain}
\newtheorem{thm}{Theorem}[section]
\newtheorem{lem}[thm]{Lemma}
\newtheorem{cor}[thm]{Corollary}

\theoremstyle{definition}

\newtheorem*{rk}{Remark}

\begin{document}
 
\thispagestyle{empty}

\title{The Expected Time to End the Tug-of-War in a Wedge}

\author{
Dante DeBlassie\\
Robert G. Smits\\
Department of Mathematical Sciences\\
New Mexico State University\\
P. O. Box 30001 \\
Department 3MB\\
Las Cruces, NM \ 88003-8001\\
deblass@nmsu.edu\\
rsmits@nmsu.edu}

\date{}
\maketitle

\begin{quote}
Using a solution of a nonhomogeneous partial differential equation involving the $p$-Laplacian, we study the finiteness of the expected time to end the tug-of-war in a wedge.
\end{quote}

\bigskip 

\n \emph{2000 Mathematical Subject Classification}. Primary 60G40, 60K99, 91A15, 35J92. Secondary 91A24, 60G42, 35B65, 34A34.
\bigskip 

\n \emph{Key Words and Phrases}. Tug-of-war, wedge, $p$-harmonic functions, inhomogeneous game $p$-Laplacian; expected time to end the game, critical angle.

\bigskip
\n \emph{Running Title: }Expected Time to End the Tug-of-War. 
\newpage

\section{Introduction}\label{sec1}

\indent 

For $d\geq 2$, let $D\subseteq\bb{R}^d$ be a domain. Consider the two-player, zero-sum game in $D$ known as the \emph{tug-of-war with noise}. Here is a rough description---we will be precise below: Suppose $\vp>0$ and $p>1$. The game starts at a position $x\in D$. A fair coin is tossed at each stage of the game and the winner picks $v\in\bb{R}^d$ with $|v|\leq\vp$ to add to the game position. Then a random noise vector with mean $0$ and variance $\frac pq\vp^2$ (where $q$ is the conjugate of $p$) in each orthogonal direction is added to the game position. The game terminates when the position reaches a point on the boundary of $D$.

We have in mind that player I seeks to prolong the game as long as possible while player II wants to end the game as soon as possible. The goal of this article is to see how the geometry of the domain affects the expected time to end the game. In particular, we study the special case of a wedge in two dimensions.

Denote the wedge in $\bb{R}^2$ with angle $\eta\in(0,2\pi)$ by

\[
W_\eta=\left\{(r,\theta):r>0,\quad-\frac\eta 2<\theta<\frac\eta 2\right\},
\]
where $r$ and $\theta$ are the usual polar coordinates. In the sequel, we use $E_x$ to denote expectation associated with the game starting at the position $x$.

\begin{thm}\label{thm1.1}
If
\[
\eta<\pi\left[1-\frac 12\sqrt{\frac{2(p-1)}p}\ \right],
\]
then there is $\vp_0>0$ along with a strategy for player II such that the time $\tau$ to end the game in $W_\eta$ satisfies 
\[
\sup_{S_I}\sup_{0<\vp<\vp_0}\vp^2E_{x_0}[\tau]<\infty,\quad x_0\in W_\eta,
\]
where the first supremum is taken over all strategies for player I.
\end{thm}

Suppose the state space of the game is a bounded set $D\subseteq\bb{R}^d$ and $F:\partial D\to\bb{R}$ is continuous. Peres and Sheffield (2008) have studied the tug-of-war in $D$, where the game is run so that when it ends at position $y\in\partial D$, player I receives a payoff of $F(y)$ from player II. Here $F$ can take on positive or negative values, and so a negative payoff corresponds to player I paying player II. An important idea in that article is a connection between the game and the game $p$-Laplacian, which is the operator defined by
\[
\Delta_p\,u:=\Delta_p^Gu=\frac 1p\,|\nabla u|^{2-p}\, \text{div}(|\nabla u|^{p-2}\,\nabla u).
\]
One of their principal results is that if $D$ is bounded and sufficiently regular, then as $\vp\to 0$, the expected payoff for player I converges to the unique $p$-harmonic extension of $F$ to all of $D$. One can regard this result as an analogue of Kakutani's (1944) classical theorem that if $B_t$ is Brownian motion in $\bb{R}^d$ and $\tau$ is its exit time from $D$, then $E_x[F(B_\tau)]$ is the unique harmonic extension of $F$ to $D$.

Another very interesting result of Peres and Sheffield is the following. Suppose $u(x)$ is sufficiently regular and satisfies $\Delta_p\,u=-g$ in $D$, where $g$ is bounded below by a positive constant. Modify the tug-of-war so that player I receives a running payoff of $\vp^2f(x_k)$ when the game position at the $k^\text{th}$ step is $x_k$. Here, $f$ is proportional to $g$ and the constant of proportionality depends only on $p$ and the underlying noise. Then as $\vp\to 0$, the expected payoff for player I converges to $u$. This particular connection is why one uses the game $p$-Laplacian rather than the usual variational $p$-Laplacian given by
\[
\Delta_p^Vu=\text{div}\left(|\nabla u|^{p-2}\,\nabla u\right).
\]
Note that when $g\equiv 1$ and the boundary payoff $F$ is zero, if $\tau$ is the time to end the game, then the expected payoff is proportional to $\vp^2E_x[\tau]$. Thus as $\vp\to 0$, the limiting value of $\vp^2E_x[\tau]$ is proportional $u$.

This observation and the following analogue for Brownian motion are what motivated our work. For the exit time $\tau_D$ of $d$-dimensional Brownian motion from a Greenian domain $D$, it is well-known that if $G(x,y)$ is Green's function for half the Dirichlet Laplacian on $D$, then
\[
E_x[\tau_D]=\int_DG(x,y)\,dx
\]
(see Hunt (1956), page 309). Moreover, when $D$ is bounded, under certain mild conditions on $\partial D$ (see Dynkin and Yushkevich (1969), page 68), the function
\[
u(x)=E_x[\tau_D]
\]
is the unique solution of the boundary value problem
\begin{align}\label{eq1.1}
\tfrac 12\Delta u&=-1\quad\text{in $D$}\notag\\
u&=0\quad\text{on $\partial D$.}
\end{align}
Thus if one can solve \eqref{eq1.1}, then it follows that
\begin{equation}\label{eq1.2}
E_x[\tau_D]<\infty.
\end{equation}
When $D$ is unbounded, the situation is more delicate, but it is not hard to show that if a nonnegative solution to \eqref{eq1.1} exists, then by looking at bounded subdomains and using the maximum principle, \eqref{eq1.2} holds. Thus the connection between the $p$-Laplacian and the tug-of-war suggests that study of the equation $\Delta_pu=-1$ might yield information on the expected time to end the game. In fact, in order to prove Theorem \ref{thm1.1}, we will make use of the following theorem.

\begin{thm}\label{thm1.2}
Let $p\in(1,\infty)$. If
\[
\eta<\pi\left[1-\frac 12\sqrt{\frac{2(p-1)}p}\ \right],
\]
then the boundary value problem
\begin{align*}
\Delta_p\,u&=-1\quad\text{in }W_\eta\\
u&=0\quad\text{on }\partial W_\eta
\end{align*}
has a nonnegative solution $u\in C^3(\ovl{W_\eta}\backslash\{0\})\cap C(\ovl{W_\eta})$ of the form $u(x)=r^2f(\theta)$. Moreover, $u$ is positive on $W_\eta$ and $|\nabla u|\neq 0$ on $\ovl{W_\eta}\backslash\{0\}$.
\end{thm}

\begin{rk}
When $p=1$, we can get an implicit solution of the corresponding boundary value problem in terms of elementary functions.
\end{rk}
Although the purpose of this theorem is to prove Theorem \ref{thm1.1}, we feel it is of independent interest because it concerns a nonhomogeneous boundary value problem involving the $p$-Laplacian in an unbounded set.

Several authors have studied nonhomogeneous partial differential equations involving the $p$-Laplacian in unbounded domains. For instance:
\begin{itemize}
\item Liouville-type theorems and related results were obtained in the articles by Liskevich et al. (2007), Bidaut-V\'eron (1989) and Abdellaoui and Peral (2003).
\item Eigenvalue problems on $\bb{R}^n$ were studied by Brown and Reichel (2004) for $L^q(\bb{R}^n)$ eigenfunctions, while both Fleckinger et al. (1997) and Dr\'abek (1995) considered positive eigenfunctions that decay to $0$ at infinity.
\item Eigenvalue problems on smooth unbounded domains with nonhomogeneous boundary conditions and eigenfunctions in weighted Sobolev spaces were the  subject of the articles by Montefusco and R\u{a}dulescu (2001) and Pfl\"{u}ger (1998). Fleckinger et al. (1999) looked at the principal eigenvalue with $L^q$ principal eigenfunction for Dirichlet boundary conditions.
\item In exterior domains, Yu (1992) considered decaying solutions of nonhomogeneous equations with Dirichlet boundary conditions.
\item Krist\'aly (2004) considered nonhomogeneous systems involving the p-Laplacian in unbounded strips with Dirichlet boundary conditions.
\end{itemize}
In our result, we consider Dirichlet boundary conditions and explosion at infinity.

Our proof of Theorem \ref{thm1.2} shows that if $\eta$ is such that the boundary value problem
\begin{align}\label{eq1.3}
\Delta_p\,u&=-1\quad\text{in $W_\eta$}\notag\\
u&=0\quad\text{on $\partial W_\eta$,}
\end{align}
has a solution in $C^3(W_\eta)\cap C(\ovl{W_\eta})$ of the form $u(x)=r^2f(\theta)$, then necessarily
\[
\eta<\pi\left[1-\frac 12\sqrt{\frac{2(p-1)}p}\ \right].
\]
But neither this nor Theorem \ref{thm1.2} is helpful in determining whether or not $E_x[\tau_D]=\infty$. 

With the aid of Theorem \ref{thm1.1}, we can prove the following result.

\begin{thm}\label{thm1.3}
There is a critical angle $\eta_p$ with the following properties:

\n i) If $\eta<\eta_p$, then there exists a strategy for player II such that if $\tau$ is the time to end the game in $W_\eta$, then
\[
E_{x_0}[\tau]<\infty,\quad x_0\in W_\eta,
\]
regardless of the strategy used by player I.

\n ii) If $\eta>\eta_p$, then for each strategy of player II, there is a strategy for player I such that
\[
E_{x_0}[\tau]=\infty,\quad x_0\in W_\eta.
\]
\end{thm}

An immediate corollary of Theorem \ref{thm1.1} is a lower bound on $\eta_p$:

\begin{cor}\label{cor1.4} 
The critical angle $\eta_p$ from Theorem \ref{thm1.3} satisfies
\[
\eta_p\geq\pi\left[1-\frac 12\sqrt{\frac{2(p-1)}p}\ \right].
\]
\end{cor}

The next result is a complement to this.

\begin{thm}\label{thm1.5}
i) The critical angle $\eta_p$ satisfies $\eta_p\leq\pi$.

\n ii) For nonconvex wedges $W_\eta$, there is a strategy for player I such that
\[
E_{x_0}[\tau]=\infty,\quad x_0\in W_\eta,
\]
for every strategy of player II.
\end{thm}

There is an interesting connection with our lower bound on $\eta_p$ and results of Aronsson (1986). His results can be shown to imply that there is $\wt{\eta_p}$ such that the boundary value problem
\begin{align*}
\Delta_p\,u&=0\quad\text{in}\quad W_{\wt{\eta_p}}\\
u&=0\quad\text{on}\quad\partial W_{\wt{\eta_p}}
\end{align*}
has a solution $u\in C\left(\ovl{W_{\wt{\eta_p}}}\right)\cap C^\infty\left(\ovl{W_{\wt{\eta_p}}}\backslash\{0\}\right)$, positive on $W_{\wt{\eta_p}}$, with the form $u(x)=r^2h(\theta)$. In fact,
\[
\widetilde{\eta_p}=\pi\left[1-\frac 12\sqrt{\frac{2(p-1)}p}\ \right],
\]
which is exactly our lower bound on $\eta_p$. For Brownian motion, it is easy to show that the function
\[
u(x)=r^2\cos 2\theta
\]
satisfies $u>0$ on $W_{\pi/2}$, $u\in C\left(\ovl{W_{\pi/2}}\right)\cap C^\infty\left(\ovl{W_{\pi/2}}\backslash\{0\}\right)$ and
\begin{align*}
\Delta u&=0\quad\text{in}\quad W_{\pi/2}\\
u&=0\quad\text{on}\quad\partial W_{\pi/2}.
\end{align*}
Results and methods of Davis and Zhang (1994) or Burkholder (1977) can be used to show that
\[
E_x\left[\tau_{W_\eta}\right]<\infty\quad\text{iff}\quad\eta<\frac{\pi}2.
\]
This leads us to conjecture that our lower bound on $\eta_p$ is sharp: that is,
\[
\eta_p=\pi\left[1-\frac 12\sqrt{\frac{2(p-1)}p}\ \right].
\]
Our method is not refined enough to make this determination.

In the case of Brownian motion, there is more known about the exit time from unbounded domains. In the case of axially symmetric cones in $\bb{R}^d$ ($d\geq 2)$, Burkholder (1977) showed there is a critical angle for the cone in which the $p^\text{th}$ moment of the exit time is finite. This result was extended to conditioned Brownian motion by Davis and Zhang (1994). DeBlassie (1987) and Ba\~nuelos and Smits (1997) found series expansions for $P_x(\tau_D>t)$ for very general cones. The case of conditioned Brownian motion was also covered in the latter reference. The series expansions immediately show that $P_x(\tau_D>t)$ decays as a power of $t$, were the power depends on the geometry of the cone. For the parabolic domain
\[
D=\{(x,y)\in\bb{R}^2:x>0,\ |y|<x^{1/2}\},
\]
Ba\~nuelos et al. (2001) showed that
\[
P_x(\tau_D>t)\approx e^{-ct^{1/3}}
\]
for large values of $t$. Our results shed some light on the corresponding situation for the tug-of-war. Since the domain
\[
D_{A,\gamma}=\{ \{(x,y)\in\bb{R}^2:x>0,\ |y|<Ax^\gamma\},\quad A>0,\ 0<\gamma<1
\]
is contained in wedges of arbitrarily small aperture, Theorem \ref{thm1.1} immediately yields the following corollary.
\begin{cor}\label{cor1.6}
For the domain $D_{A,\gamma}$, there exists a strategy for player II and $\vp_0>0$ such that
\[
\sup_{S_I}\sup_{0<\vp<\vp_0}\vp^2E_{x_0}[\tau]<\infty,\quad x_0\in D_{A,\gamma}.
\]
\end{cor}

The article is organized as follows. In section two we give the rigorous definition of the tug-of-war. We also summarize some results of Peres and Sheffield and give our fundamental computational tool. Section three is devoted to the proof of Theorem \ref{thm1.1}, using Theorem \ref{thm1.2}. In section four we prove Theorem \ref{thm1.5}, making use of some ideas of Burkholder (1977). Then we prove Theorem \ref{thm1.3}. Finally, in section five we prove Theorem \ref{thm1.2}.

\section{Preliminaries}\label{sec2}

\indent 

Let $D\subseteq\bb{R}^d$ be open and connected. The \emph{noise measure} $\mu$ is a compactly supported mean zero Borel probability measure on $\bb{R}^d$ which is preserved by orthogonal transformations of $\bb{R}^d$ that fix the first basis vector ${\bf e}_1$. For each $v\in\bb{R}^d$, let $\Psi$ be $|v|$ times some orthonormal transformation of $\bb{R}^d$ chosen so that $\Psi({\bf e}_1)=v$. Define a new probability measure on the Borel sets of $\bb{R}^d$ by
\[
\mu_v(B)=\mu\left(\Psi^{-1}(B)\right).
\]
Since $\mu$ is invariant under orthogonal transformations of $\bb{R}^d$ which fix ${\bf e_1}$, $\mu_v$ is independent of the choice of $\Psi$. For $R>0$ and $z\in\bb{R}^d$, let
\[
B_R(z)=\{x\in\bb{R}^d:|z-x|<R\}
\]
and set
\[
\alpha=1+\inf\{R:\mu(B_R(0))=1\}.
\]
The tug-of-war in $D$, with noise $\mu$, is played as follows. Let $x_0\in D$ be the initial game position. At the $k^\text{th}$ turn, a fair coin is tossed.
\begin{itemize}
\item If $\text{dist}(x_{k-1},\partial D)>\alpha\vp$, then the winning player chooses $v_k\in\bb{R}^d$ with $|v_k|\leq\vp$ and the game position is moved to
\[
x_k=x_{k-1}+v_k+z_k,
\]
where $z_k$ is a random noise vector sampled from $\mu_{v_k}$.
\item If $\text{dist}(x_{k-1},\partial D)\leq \alpha\vp$, then the winning player chooses $x_k\in\partial D$ with $|x_k-x_{k-1}|\leq \alpha\vp$ and the game ends.
\end{itemize}
This is a basic description of the game movement. There are are many possible choices of payoffs; for instance,
\begin{itemize}
\item the payoff can occur only when the game ends;
\item there is a running payoff at each stage of the game;
\item the payoff is a combination of the two.
\end{itemize}
Also, there are related games that have been studied. See Peres et al. (2009), Peres et al. (2007), Lazarus et al. (1996), Maitra and Sudderth (1998) and Spencer (1977). 

Let $\pi_j$ be the projection to the $j^\text{\,th}$ coordinate and use $C=\{C_{ij}\}$ to denote the covariance matrix of $\mu$:
\[
C_{ij}=\int\pi_i(x)\pi_j(x)d\mu(x).
\]
Note that $C$ is diagonal and $C_{ii}=C_{jj}$ for $2\leq i,j\leq d$. Define
\[
p=p(\mu)=\frac{C_{11}+C_{22}+1}{C_{22}}.
\]
Then for some $\beta>0$, with $q$ being the conjugate of $p$, we have
\[
C_{11}+1=\frac\beta q
\]
\[
C_{ii}=\frac\beta p,\quad 2\leq i\leq d.
\]
The \emph{game history up to step} $k$ is a sequence of moves
\[
h_k=(x_0,v_1,x_1,v_2,x_2,\dots,v_k,x_k)
\]
regarded as an element of the set
\[
H_k=D\times\left(\ovl{B_\vp(0)}\times\ovl D\right)^k;
\]
we use the convention that if the game terminated at time $j<k$, then $v_m=0$ and $x_m=x_j$ for $m\geq j$. The \emph{complete history space} $H_\infty$ is the set of all infinite position sequences
\[
h=(x_0,v_1,x_1,v_2,x_2,\dots)
\]
endowed with the product topology. A \emph{strategy} is a sequence of Borel measurable maps from $H_k$ to $\ovl{B_\vp(0)}$ giving the move a player would make at the $k^\text{th}$ step as a function of the game history.

A pair of strategies $(S_I,S_{II})$ for players I and II, respectively, and a starting point $x_0=x$ determine a unique Borel probability measure on $H_\infty$. We will use $E_x$ to denote the corresponding expectation.

In what follows:
\begin{itemize}
\item $x\in\bb{R}^d$ will be regarded as a column vector;
\item $x^T$ will be its transpose;
\item $(x,y)$ will be the usual Euclidean inner product of $x,y\in\bb{R}^d$;
\item $\displaystyle ||A||=\sup_{|v|\leq 1}|Av|$ will be norm of the $d\times d$ matrix $A$.
\end{itemize}

The infinity Laplacian operator is defined by
\[
\Delta_\infty u=|\nabla u|^{-2}\sum_{i,j}\frac{\partial u}{\partial x_i}\,\frac{\partial^2u}{\partial x_i\partial x_j}\,\frac{\partial u}{\partial x_j}
\]
when $u$ is sufficiently regular. Then the $p$-Laplacian can be expressed as
\begin{equation}\label{eq2.1}
\Delta_p\,u=\frac 1p\Delta u+\left(\frac 1q-\frac 1p\right)\Delta_\infty u,
\end{equation}
where $\Delta=\Delta_2$ is the usual Laplacian.

The following facts are proved in Peres and Sheffield (2008). Given a symmetric $d\times d$ matrix $A$ and $\xi\in\bb{R}^d\backslash \{0\}$, define
\[
\phi(x)=x^TAx+(\xi,x)
\]
and
\begin{align}\label{eq2.2}
\psi(v)&=E_0[\phi(x_1)|\text{ player I won and chose $v=v_1$ }]\notag\\
&=E_0[\phi(v_1+z_1)|\text{ player I won and chose $v=v_1$ }].
\end{align}
Note we also have
\begin{equation}\label{eq2.3}
\psi(v)=E_0[\phi(v_1+z_1)|\text{ player II won and chose $v=v_1$ }].
\end{equation}
Then for
\begin{equation}\label{eq2.4}
B=\left(\frac\beta q-\frac\beta p\right)A+\frac\beta p(\text{Tr\,}A)I,
\end{equation}
we have
\begin{equation}\label{eq2.5}
\psi(v)=(\xi,v)+v^TBv,
\end{equation}
\begin{equation}\label{eq2.6}
\psi\left(\frac{\vp\xi}{|\xi|}\right)=\vp|\xi|+\frac 12\Delta_\infty\psi(0)\,\vp^2,
\end{equation}
and for $\vp<1$, if $v_\text{MAX}\in\ovl{B_\vp(0)}$ is such that $\psi(v_\text{MAX})$ is maximal, then
\begin{equation}\label{eq2.7}
\left|\psi(v_\text{MAX})-\vp|\xi|-\frac 12\Delta_\infty\psi(0)\,\vp^2\right|\leq\frac{16||B||^2}{|\xi|}\,\vp^3.
\end{equation}

Since $-\psi$ is obtained from $\psi$ by replacing $\xi$ and $B$ by $-\xi$ and $-B$, respectively, the analogue of \eqref{eq2.6} for $-\psi$ is
\begin{equation}\label{eq2.8}
\psi\left(-\frac{\vp\xi}{|\xi|}\right)=-\vp|\xi|+\frac 12\Delta_\infty\psi(0)\,\vp^2.
\end{equation}

Note that by \eqref{eq2.4}--\eqref{eq2.5} and \eqref{eq2.1},
\begin{align}\label{eq2.9}
\Delta_\infty\psi(0)&=\left(\frac\beta q-\frac\beta p\right)\Delta_\infty\phi(0)+\frac\beta p\Delta\phi(0)\notag\\
&=\beta\Delta_p\,\phi(0).
\end{align}

\begin{lem}\label{lem2.1}
Let $A$ be a symmetric $d\times d$ matrix and let $\xi\in\bb{R}^d\backslash \{0\}$. Fix $k\geq 0$ and let
\[
h_k=(x_0,v_1,x_1,v_2,x_2,\dots,v_k,x_k)
\]
be the game history up to step $k$. Set
\[
\phi_1(x)=(x-x_k)^TA(x-x_k)+(\xi,x-x_k),\quad x\in\bb{R}^d.
\]

i) Suppose at move $k+1$, player I adds $v$ to the game position if he wins and player II adds $z$ to the game position if he wins. Then
\[
E_{x_0}[\phi_1(x_{k+1})|h_k]=\tfrac 12\psi(v)+\tfrac 12\psi(z).
\]

ii) If player II's strategy at move $k+1$ is to tug $\vp$ units in the direction of $-\xi$, then for any strategy of player I,
\[
E_{x_0}[\phi_1(x_{k+1})|h_k]\leq\frac{M}{|\xi|}\vp^3+\frac\beta 2\Delta_p\,\phi_1(x_k)\,\vp^2,
\]
where $M=8\beta^2(d+1)^2||A||^2$.
\end{lem}

\begin{proof} 
By translation invariance of both the game and the infinity Laplacian, for
\begin{align*}
\psi_1(v)&=E_{x_0}[\phi_1(x_{k+1})|\,h_k,\text{ player I won at move $k+1$ and chose $v_{k+1}=v$ }]\\
&=E_{x_0}[\phi_1(x_{k+1})|\,h_k,\text{ player II won at move $k+1$ and chose $v_{k+1}=v$ }],
\end{align*}
we have
\begin{equation}\label{eq2.10}
\psi_1(v)=(\xi,v)+v^TBv\quad\text{(by \eqref{eq2.2} and \eqref{eq2.5})},
\end{equation}
and
\begin{align}\label{eq2.11}
\psi_1\left(-\frac{\vp\xi}{|\xi|}\right)&=-\vp|\xi|+\frac 12\Delta_\infty\psi_1(x_k)\,\vp^2\quad\text{(by \eqref{eq2.8})}\notag\\
&=-\vp|\xi|+\frac \beta 2\Delta_p\,\phi_1(x_k)\,\vp^2\quad\text{(by \eqref{eq2.9}).}
\end{align}
Moreover, for $\vp<1$, by \eqref{eq2.7} and \eqref{eq2.9} we have
\begin{equation}\label{eq2.12}
\left|\psi_1(v_\text{MAX})-\vp|\xi|-\frac\beta 2\Delta_p\,\phi_1(x_k)\,\vp^2\right|\leq\frac{16||B||^2}{|\xi|}\,\vp^3,
\end{equation}
where $v_\text{MAX}\in \ovl{B_\vp(0)}$ is such that $\psi_1(v_\text{MAX})$ is maximal.

If $h_k^I$ indicates that at move $k+1$, player I chooses $v_{k+1}=v$ if he wins and if $h_k^{II}$ indicates that at move $k+1$, player II chooses $v_{k+1}=z$ if he wins, then we have
\begin{align}\label{eq2.13}
E_{x_0}[\phi_1(x_{k+1})|h_k]&=\tfrac 12E_{x_0}\left[\phi_1(x_{k+1})|h_k^{I}\right]+\tfrac 12E_{x_0}\left[\phi_1(x_{k+1})|h_k^{II}\right]\notag\\
&=\tfrac 12\psi_1(v)+\tfrac 12\psi_1\left(z\right)\\
&=\tfrac 12\psi(v)+\tfrac 12\psi(z)\quad\text{(by \eqref{eq2.10} and \eqref{eq2.5}).}\notag
\end{align}
This yields part i) of the lemma.

As for part ii), the formula \eqref{eq2.13} implies that
\begin{align*}
E_{x_0}[\phi_1(x_{k+1})|h_k]&=\frac 12\psi_1(v)+\frac 12\psi_1\left(-\frac{\vp\xi}{|\xi|}\right)\\
&\leq\frac 12\left[\vp\,|\xi|+\frac\beta 2\Delta_p\,\phi_1(x_k)\,\vp^2+\frac{16||B||^2}{|\xi|}\,\vp^3\right]\\
&\qquad +\frac 12\left[-\vp\,|\xi|+\frac\beta 2\Delta_p\,\phi_1(x_k)\,\vp^2\right]
\intertext{(by \eqref{eq2.12} and \eqref{eq2.11}, respectively)}
&=\frac{8||B||^2}{|\xi|}\,\vp^3+\frac\beta 2\Delta_p\,\phi_1(x_k)\,\vp^2\\
&\leq\frac{M}{|\xi|}\,\vp^3+\frac\beta 2\Delta_p\,\phi_1(x_k)\,\vp^2\qquad\text{(by \eqref{eq2.4}),}
\end{align*}
as desired.
\end{proof}


\section{Proof of Theorem \ref{thm1.1}}\label{sec3}

\indent 

Let 
\[
\eta<\pi\left[1-\frac 12\sqrt{\frac{2(p-1)}{p}}\ \right].
\]
Choose $\eta_1\in\left(\eta,\pi\left[1-\frac 12\sqrt{\frac{2(p-1)}{p}}\,\right]\right)$ and let $u(x)=r^2f(\theta)$ be from Theorem \ref{thm1.2} for the wedge $W_{\eta_1}$. The tug-of-war is translation invariant, so to prove Theorem \ref{thm1.1}, it suffices to show that if $\tau_W$ is the time to end the game in the translated wedge
\[
W=W_\eta+2(\alpha +1),
\]
(recall $\alpha$ us from the beginning of section 2), then there is a strategy for player II and $\vp_0>0$ such that
\begin{equation}\label{eq3.1}
\sup_{S_I}\sup_{0<\vp<\vp_0}\vp^2\,E_{x_0}[\tau_W]<\infty,\quad x_0\in W,
\end{equation}
where the supremum is taken over all strategies for player I.

We have $u\in C^3(\ovl{W})$ and its second and third order partials are of the forms
\[
\sum_{j=0}^2k_j(\theta)f^{(j)}(\theta)
\]
\[
\frac 1r\sum_{j=0}^3\ell_j(\theta)f^{(j)}(\theta),
\]
respectively, where the $k_j$'s and $\ell_j$'s are bounded. Thus for all $i,j,k$,
\begin{equation}\label{eq3.2}
\sup_{\ovl W}\left|\frac{\partial^2u}{\partial x_i\partial x_j}\right|
\leq\sup\left\{\,\left|\frac{\partial^2u}{\partial x_i\partial x_j}\right|: x\in W_\eta\text{ and }|x|\geq \alpha+1\right\}<\infty
\end{equation}
and
\begin{equation}\label{eq3.3}
\sup_{\ovl W}\left|\frac{\partial^3u}{\partial x_i\partial x_j\partial x_k}\right|
\leq\sup\left\{\,\left|\frac{\partial^3u}{\partial x_i\partial x_j\partial x_k}\right|: x\in W_\eta\text{ and }|x|\geq\alpha+1\right\}<\infty.
\end{equation}
These bounds are the reason why we must use the $u$ corresponding to $W_{\eta_1}$ and translate $W_\eta$ to $W$---they do not hold in $W_\eta$ if we use the function corresponding to $W_\eta$. A simple computation shows that
\[
|\nabla u|^2=r^2\left[4f^2+(f')^2\right],
\]
and since $|\nabla u|>0$ on $\ovl{W_{\eta_1}}\backslash\{0\}$, we have tthat
\begin{equation}\label{eq3.4}
\inf_{\ovl W}|\nabla u|>0.
\end{equation}
Thus for $\vp<\frac 12$, given any $y\in\ovl W$, for $\gamma=2(\alpha+1)$, we have a Taylor expansion
\begin{equation}\label{eq3.5}
u(x)=u(y)+(\nabla u(y),x-y)+\frac 12(x-y)^TD^2u(y)(x-y)+R(x,y),\quad x\in B_{\gamma\vp}(y),
\end{equation}
where $D^2u$ is the matrix of second order partials of $u$ and the remainder $R(x,y)$ satisfies
\begin{equation}\label{eq3.6}
|R(x,y)|\leq C\vp^3,
\end{equation}
with $C$ independent of $x,y$ and $\vp<\frac 12$.

\bigskip
\n{\bf Proof of \eqref{eq3.1}.} Define 
\begin{equation}\label{eq3.7}
C_1:=C+18\beta^2\sup_{\ovl{W}}\frac{||D^2u||^2}{|\nabla u|^2},
\end{equation}
where $C$ is from \eqref{eq3.6}. Note that $C_1<\infty$ by \eqref{eq3.2} and \eqref{eq3.4}. For $k\geq 0$,  define
\[
M_k=u(x_k)+\frac\beta 2\,\vp^2\,k-C_1k\,\vp^3,
\]
where $x_k$ is the game position at step $k$ for the tug-of-war in $W$. Then $M_k$ is a supermartingale: Indeed, by the Taylor expansion \eqref{eq3.5}, we have
\begin{align}\label{eq3.8}
E_{x_0}\left[M_{k+1}-M_k|h_k\right]&=E_{x_0}\left[u(x_{k+1})-u(x_k)+\frac\beta 2\,\vp^2-C_1\vp^3\,\bigg|\,h_k\,\right]\notag\\
&\leq E_{x_0}\left[(\nabla u(x_k),x_{k+1}-x_k)+\frac 12(x_{k+1}-x_k)^TD^2u(x_k)(x_{k+1}-x_k)\right.\notag\\
&\hspace{1in} \left.+C\vp^3+\frac\beta 2\,\vp^2-C_1\vp^3\,\bigg|\,h_k\,\right]\notag\\
&= E_{x_0}\left[(\nabla u(x_k),x_{k+1}-x_k)+\frac 12(x_{k+1}-x_k)^TD^2u(x_k)(x_{k+1}-x_k)\right.\notag\\
&\hspace{1in} \left.+\left(C-C_1\right)\vp^3+\frac\beta 2\,\vp^2\bigg|h_k\right].
\end{align}
For $\xi=\nabla u(x_k)$ and $A=\frac 12D^2u(x_k)$ in Lemma \ref{lem2.1}, a simple calculation shows that the corresponding $\phi_1$ satisfies 
\begin{equation}\label{eq3.9}
\Delta_p\,\phi_1(x_k)=\Delta_p\,u(x_k)=-1.
\end{equation}
Now assume that player II uses the strategy such that at the $k^\text{th}$ position $x_k$, he moves $\vp$ units in the direction of $-\nabla u(x_k)$ if he wins and player I uses any strategy if he wins. Then by part i) of Lemma \ref{lem2.1}, for
\[
M=18\beta^2||D^2u(x_k)||^2,
\]
we have from \eqref{eq3.8} that
\begin{align*}
E_{x_0}[M_{k+1}-M_k|h_k]&\leq E_{x_0}[\phi_1(x_{k+1})|h_k]+\left(C-C_1\right)\,\vp^3+\frac\beta 2\,\vp^2\\
&\leq \frac M{|\nabla u(x_k)|}\,\vp^3+\frac\beta 2\Delta_p\,\phi_1(x_k)\,\vp^2+\left(C-C_1\right)\,\vp^3+\frac\beta 2\,\vp^2\\
&=\frac M{|\nabla u(x_k)|}\,\vp^3-\frac\beta 2\,\vp^2+\left(C-C_1\right)\,\vp^3+\frac\beta 2\,\vp^2\\
\intertext{(by \eqref{eq3.9})}
&\leq 0,
\end{align*}
by choice of $C_1$ from \eqref{eq3.7}. Thus $M_k$ is a supermartingale, as claimed.

By optional stopping, for $k\geq 0$,
\begin{align*}
u(x_0)=M_0&\geq E_{x_0}[M_{k\wedge\tau_W}]\\
&=E_{x_0}\left[u(x_{k\wedge\tau_W})+\frac\beta 2\vp^2\,(k\wedge\tau_W)-C_1\vp^3\,(k\wedge\tau_W)\right]\\
&\geq\vp^2\,\left[\frac\beta 2-C_1\vp\right]E_{x_0}[k\wedge\tau_W].
\end{align*}
Taking $\vp_0\in\left(0,\frac 12\right)$ so small that $\frac\beta 2-C_1\vp>0$ for all $\vp<\vp_0$, we can use monotone convergence on the left to end up with
\[
\vp^2\,\left[\frac\beta 2-C_1\vp\right]E_{x_0}[\tau_W]\leq u(x_0).
\]
Then
\[
\sup_{0<\vp<\vp_0}\vp^2E_{x_0}[\tau_W]<\infty,\quad x_0\in W,
\]
giving \eqref{eq3.1}.

\bigskip
This completes the proof of Theorem \ref{thm1.1}.\hfill$\square$


\section{Proof of Theorem \ref{thm1.5} and Theorem \ref{thm1.3}}\label{sec4}

Theorem \ref{thm1.5} is an immediate consequence of the following theorem.
\begin{thm}\label{thm4.1}
For the tug-of-war in $W_\pi$, there is a strategy for player I such that the time $\tau$ to end the game satisfies
\[
E_{x_0}[\tau]=\infty,\quad x_0\in W_\pi,
\] 
for any strategy used by player II.
\end{thm}
\begin{proof}
Let $u(x)$ be the projection $\pi_1(x)$ onto the first coordinate. Suppose player I's strategy at play $k+1$ is to move $\vp$ units in the direction of $\nabla u(x_k)$ if he wins. To get a contradiction, assume there is a strategy for player II such that
\begin{equation}\label{eq4.1}
E_{x_0}[\tau]<\infty.
\end{equation}
For $\xi=\nabla u(x_k)$ and $A=\frac 12D^2u(x_k)=0$ in Lemma 2.1, we see the corresponding matrix $B$ from \eqref{eq2.4} satisfies $B=0$ and it is easy to show that the corresponding $\phi_1$ satisfies
\[
\Delta_p\,\phi_1(x_k)=0.
\]
Then by part i) of the Lemma, with $v$ denoting the move player II makes at move $k+1$ when he wins,
\begin{align*}
E_{x_0}\left[u\left(x_{k+1}\right)-u\left(x_k\right)|\,h_k\right]&=E_{x_0}\left[(\nabla u(x_k),x_{k+1}-x_k)+\frac 12(x_{k+1}-x_k)^TD^2u(x_k)(x_{k+1}-x_k)\,\bigg|\,h_k\,\right]\notag\\
&=\frac 12\psi_1\left(\frac{\vp\xi}{|\xi|}\right)+\frac 12\psi_1(v)\notag\\
&=\frac 12\left[\left(\xi,\frac{\vp\xi}{|\xi|}\right)+0\right]+\frac 12\left[\left(\xi,v\right)+0\right]\\
&=\frac 12\left[\vp|\xi|+(\xi,v)\right]\\
&\geq 0,
\end{align*}
since $|v|\leq\vp$. Thus $u(x_k)$ is a submartingale.

Let $u\left(x_\tau\right)^*$ be the maximal function defined by
\[
u\left(x_\tau\right)^*=\sup_{k\leq\tau}u(x_k).
\]

Recalling that $\gamma=2(\alpha+1)$, choose $\beta_1>1+\gamma\vp$ and then let $\delta\in(0,1)$ be so small that
\begin{equation}\label{eq4.2}
\frac{\beta_1\gamma\vp\delta}{\beta_1-1-\gamma\vp}<1.
\end{equation}
With the good-$\lambda$ inequalities of Burkholder (1973) in mind, we now show that
\begin{equation}\label{eq4.3}
P_{x_0}\left(u(x_\tau)^*\geq\beta_1\lambda, \tau+u(x_0)\leq\delta\lambda\right)\leq\frac{\gamma\vp\delta}{\beta_1-1-\gamma\vp}P_{x_0}\left(u(x_\tau)^*\geq\lambda\right),\qquad\lambda>0.
\end{equation}
To this end, note that if $u(x_0)>\delta\lambda$, then the left hand side of \eqref{eq4.3} is zero and the inequality is trivial. Thus it is no loss to assume $u(x_0)\leq\delta\lambda$. Let
\[
\xi=\inf\{k\geq 0:u(x_{\tau\wedge k})>\lambda\}.
\]
Then
\[
\lambda\leq u(x_\xi)\leq(1+\gamma\vp)\lambda\text{ on } \{\xi<\infty\}=\{u(x_\tau)^*>\lambda\}.
\]
For $a=\delta\lambda-u(x_0)$ and $[\,\cdot\,]$ denoting the greatest integer function, we have
\begin{align}\label{eq4.4}
P_{x_0}\left(u(x_\tau)^*>\beta_1\lambda,\ \tau+u(x_0)\leq \delta\lambda\right)&=P_{x_0}\left(\xi<\infty,\ \tau\leq a,\ \sup_{\xi\leq k\leq\tau}u(x_k)\geq\beta_1\lambda\right)\notag\\
&\leq P_{x_0}\left(\xi<\infty,\ \sup_{\xi\leq k\leq\xi+[\,a\,]}|u(x_k)-u(x_\xi)|\geq(\beta_1-1-\gamma\vp)\lambda\right)\notag\\
&=E_{x_0}\left[I_{\xi<\infty}P_{x_\xi}\left(\sup_{k\leq [\,a\,]}|u(x_k)-u(x_\xi)|\geq(\beta_1-1-\gamma\vp)\lambda\right)\right].
\end{align}
Now $|u(x_k)-u(x_0)|$ is a nonnegative submartingale, so by Doob's inequality,
\begin{align*}
P_x\left(\sup_{k\leq[\,a\,]}|u(x_k)-u(x)|\geq(\beta_1-1-\gamma\vp)\lambda\right)&\leq\frac 1{(\beta_1-1-\gamma\vp)\lambda}E_x\left[\,\left|u(x_{[\,a\,]})-u(x)\right|\,\right]\\
&\leq\frac 1{(\beta_1-1-\gamma\vp)\lambda}\sum_{i=1}^{[\,a\,]}E_x\left[\,\left|u(x_i)-u(x_{i-1})\right|\,\right]
\intertext{(where we take $x_0=x$ in the summation)}
&\leq\frac {[\,a\,]\gamma\vp}{(\beta_1-1-\gamma\vp)\lambda}
\intertext{(using that $u=\pi_1$)}
&\leq\frac {\delta\gamma\vp}{\beta_1-1-\gamma\vp}.
\end{align*}

Using this in \eqref{eq4.4}, we get \eqref{eq4.3}. Then by Lemma 7.1 in Burkholder (1973),
\[E_{x_0}\left[u(x_\tau)^*\right]\leq C_3E_{x_0}\left[\tau+u(x_0)\right],
\]
where
\[
C_3=\beta_1\delta^{-1}\left(1-\frac{\beta_1\delta\gamma\vp}{\beta_1-1-\gamma\vp}\right)^{-1}.
\]
Combined with \eqref{eq4.1}, this implies that the family
\[
\{u\left(x_{\tau\wedge k}\right):k\geq 0\}
\]
is uniformly integrable. By optional stopping, since $u(x_k)$ is a submartingale,
\[
u(x_0)\leq E_{x_0}\left[u\left(x_{\tau\wedge k}\right)\right]
\]
and by the uniform integrability, we can let $k\to\infty$ to get
to end up with
\[
0<u(x_0)\leq E_{x_0}\left[u\left(x_{\tau}\right)\right]=0.
\]
This yields the desired contradiction and the proof of Theorem \ref{thm4.1} is complete
\end{proof}

\n{\bf Proof of Theorem \ref{thm1.3}.} Let $\mathcal{A}$ be the set of all $\eta\in (0,2\pi)$ such that for some strategy of player II, the time $\tau$ to end the tug-of-war in $W_\eta$ satisfies
\[
E_{x_0}[\tau]<\infty,
\]
regardless of the strategy used by player I. Then it suffices to show that $\mathcal{A}\neq\emptyset$, for in that event we take $\eta_p=\sup \mathcal{A}$. But by Theorem \ref{thm1.1}, $\mathcal{A}\neq\emptyset$, as desired.\hfill$\square$


\section{Proof of Theorem \ref{thm1.2}}\label{sec5}

\indent 

The proof of Theorem \ref{thm1.2} is long and technical, so before giving the details, we motivate our argument. If $u(x)=r^2f(\theta)\geq 0$ is a $C^2$ solution to
\begin{align*}
&\Delta_p\,u=-1\text{ in }W_\eta\\
& u=0\text{ on }\partial W_\eta,
\end{align*}
then it is a routine matter to show that on $\left(-\frac\eta 2,\frac\eta 2\right)$,
\begin{equation}\label{eq5.1}
4f^2\left[1+2f+\tfrac 1pf''\right]+(f')^2\left[1+\tfrac{2(3p-4)}pf+\tfrac{p-1}pf''\right]=0.
\end{equation}
By symmetry we expect to have $f'(0)=0$. Thus it should suffice to consider \eqref{eq5.1} on $\left(0,\frac\eta 2\right)$ with
\begin{equation}\label{eq5.2}
f'(0)=0,\quad f\left(\tfrac\eta 2\right)=0.
\end{equation}
For this, set
\[
y=f\text{ and }a=f(0),
\]
and then make the transformation
\begin{equation}\label{eq5.3}
H_a(y)=(y')^2.
\end{equation}
This converts \eqref{eq5.1} into the equation
\begin{equation}\label{eq5.4}
4y^2\left[1+2y+\tfrac 1{2p}H_a'(y)\right]+H_a(y)\left[1+\tfrac{2(3p-4)}py+\tfrac{p-1}{2p}H_a'(y)\right]=0,
\end{equation}
and since
\[
H_a(a)=H_a(y(0))=(y'(0))^2,
\]
the condition \eqref{eq5.2} tells us that
\begin{equation}\label{eq5.5}
H_a(a)=0.
\end{equation}
Modulo technicalities, \eqref{eq5.3} yields an implicit representation of $y=f(\theta)$:
\begin{equation}\label{eq5.6}
\theta=\int_{y(\theta)}^a\frac{dw}{\sqrt{H_a(w)}},\qquad \theta\in\left(0,\tfrac\eta 2\right).
\end{equation}
In particular,
\begin{equation}\label{eq5.7}
\frac\eta 2=\int_{y(\eta/2)}^a\frac{dw}{\sqrt{H_a(w)}}=\int_0^a\frac{dw}{\sqrt{H_a(w)}}.
\end{equation}
Our approach is to work backwards from \eqref{eq5.4}--\eqref{eq5.5}. Thus we need to show for each $a>0$, there is a solution $H_a$ to \eqref{eq5.4} on $(0,a)$ satisfying $H_a(0)=0$. Then using \eqref{eq5.3} and \eqref{eq5.6}, we get an implicit solution $f=y$ of \eqref{eq5.1}. There will be some $\theta_a>0$ such that $y(\theta_a)=0$ and $y>0$ on $(0,\theta_a)$. In light of \eqref{eq5.7}, we need to show that for some $a>0$,
\[
\frac\eta 2=\theta_a=\int_0^a\frac{dw}{\sqrt{H_a(w)}}.
\]
This will happen if we show $\theta_a$ is continuous as a function of $a>0$ with $\theta_a\to 0$ as $a\to 0^+$ and
\[
\lim_{a\to\infty}\int_0^a\frac{dw}{\sqrt{H_a(w)}}=\frac\pi 2\left[1-\frac 12\sqrt{\frac{2(p-1)}p}\ \right].
\]
The latter is hard to prove directly. The trick is to change variables
\[
G_a(y)=(ay)^{-2}H_a(ay),\qquad 0<y<1.
\]
This converts the equation \eqref{eq5.4} to
\begin{equation}\label{eq5.8}
G'_a(y)=-\frac{8p\left[\,\frac 1a+2y\,\right]+\left[\,\frac{2p}a+4(3p-2)y+2(p-1)yG_a(y)\,\right]G_a(y)}{y^2\left[\,4+(p-1)G_a(y)\,\right]}.
\end{equation}
We will show $G_a$ is decreasing in $a$ and converges to a function $K$ which solves the equation resulting from taking the limit of \eqref{eq5.8} as $a\to\infty$. Then
\[
\lim_{a\to\infty}\int_0^a\frac{dw}{\sqrt{H_a(w)}}=\lim_{a\to\infty}\int_0^1\frac{du}{u\sqrt{G_a(u)}}=\int_0^1\frac{du}{u\sqrt{K(u)}}.
\]
The latter can be easily computed using residues, upon making an appropriate change of variables. 

We point out that \eqref{eq5.4} can be reduced to an Abel differential equation of the second kind (see Polyanin and Zaitsev (2003)). There are some equations in this family that have parametric solutions or can be solved explicitly in terms of tabulated functions, but \eqref{eq5.4} is not one of these. Panayotounakos (2005) has given exact analytic solutions, but the form corresponding to \eqref{eq5.4} is unwieldy and seems impossible to use for our purpose.

Now we give the details. For $a$, $y>0$ and $w>-\frac{p-1}4$, define
\begin{equation}\label{eq5.9}
F_a(y,w)=\frac{8p\left[\,\frac 1a+2y\,\right]+\left[\,\frac{2p}a+4(3p-2)y+2(p-1)yw\,\right]w}{y^2\left[\,4+(p-1)w\,\right]}.
\end{equation}

\begin{lem}\label{lem5.1}
For each $a>0$, there is a unique solution $G_a\in C^\infty((0,1])$, with $G_a>0$ on $(0,1)$, of the boundary value problem
\begin{equation}\label{eq5.10}
 \left\{\begin{array}{l}
G_a'(y)=-F_a(y,G_a(y))\quad 0<y<1\\
\noalign{\smallskip}
G_a(1)=0,
        \end{array}\right.
\end{equation}
and $G_a$ is continuous and strictly decreasing on $(0,1]$. Moreover, for each $y\in(0,1)$, $G_a(y)$ is a continuous function of $a>0$.
\end{lem}

\begin{rk}
When we say $G\in C^\infty((0,1])$ we mean that for some $\delta>0$, $G$ has an extension in $C^\infty((0,1+\delta))$.
\end{rk}
\begin{proof}
It is easy to check that there is some $\delta\in(0,4)$ such that if $w\geq -\delta$, then
\begin{align*}
&4+(p-1)w>0\\
\intertext{and}
&8p+2(3p-2)w+(p-1)w^2>0.
\end{align*}
In particular, for $y>0$ and $w\geq -\delta$,
\begin{equation}\label{eq5.11}
F_a(y,w)=\frac{\frac{8p+2pw}{a}+\left[16p+4(3p-2)w+2(p-1)w^2\right]y}{y^2[\,4+(p-1)w\,]}>0.
\end{equation}
Since $F_a\in C^\infty((0,\infty)\times[-\delta,\infty))$, by the Fundamental Existence/Uniqueness Theorem for ordinary differential equations (Coddington and Levinson (1955), Theorem 3.1 on page 12), for some $\delta_1>0$, there is a unique solution $G\in C^1((1-\delta_1,1+\delta_1))$ to 
\begin{align*}
&G'(y)=-F_a(y,G(y)),\quad y\in(1-\delta_1,1+\delta_1)\\
\noalign{\smallskip}
&G(1)=0.
\end{align*}
Note it is tacit that $G\geq -\delta$ on $(1-\delta_1,1+\delta_1)$. Since
\[
G'(1)=-F_a(1,G(1))=-F_a(1,0)=-\frac{2p(1+2a)}a<0,
\]
there is some $\delta_2\in(0,\delta_1)$ such that $G$ is strictly decreasing on $(1-\delta_2,1]$. Then from \eqref{eq5.11} we see that $G$ can be uniquely extended to $(0,1+\delta_1)$ (see Theorem 1.3 on page 47 in Coddington and Levinson (1955)), to be the solution in $C^1((0,1+\delta_1))$ of
\begin{equation}\label{eq5.12}
 \left\{\begin{array}{l}
G'(y)=-F_a(y,G(y))\quad 0<y<1+\delta_1\\
\noalign{\smallskip}
G(1)=0.
        \end{array}\right.
\end{equation}
Moreover, $G$ is strictly decreasing and continuous on $(0,1]$. Thus $G>0$ on $(0,1)$, and by repeatedly differentiating \eqref{eq5.12}, we see that $G\in C^\infty((0,1+\delta_1))$. Upon setting $G_a=G$, it follows that  \eqref{eq5.10} holds and $G_a$ is continuous and strictly decreasing on $(0,1]$.

Since $F_a(y,w)$ satisfies a Lipschitz condition in $w$ uniformly for $(y,w,a)$ in compact subsets of $(0,\infty)\times[0,\infty)\times(0,\infty)$, by standard theorems for ordinary differential equations involving parameters (Coddington and Levinson (1955), Theorem 7.4 on page 29), for each $y\in (0,1)$, $G_a(y)$ is a continuous function of $a>0$.
\end{proof}

The next order of business is to show $G_a$ decreases as a function of $a>0$.

\begin{lem}\label{lem5.2}
If $0<a_1<a_2$, then $G_{a_2}<G_{a_1}$ on $(0,1)$. Moreover, for each $y\in(0,1)$, $\lim_{a\to 0^+}G_a(y)=\infty$.
\end{lem}
\begin{proof}
For typographical simplicity, write $G_{a_i}=G_i$ and $F_{a_i}=F_i$, where $F_a$ is from \eqref{eq5.9}. We use $D_L\,g$ to denote the left derivative of $g$:
\[
D_L\,g(y)=\lim_{h\to 0^-}\frac{g(y+h)-g(y)}h.
\]
Since
\begin{equation}\label{eq5.13}
G_i(y)=\int_y^1F_i(u,G_i(u))\,du,\quad 0<y\leq 1,
\end{equation}
we have that
\[
D_L\,G_i(1)=-F_i(1,0)=-2p\left[\,\frac 1{a_i}+2\,\right].
\]
Consequently,
\[
D_L\,G_2(1)-D_L\,G_1(1)=2p\left[\,\frac 1{a_1}-\frac 1{a_2}\,\right]>0.
\]
Hence for small $\delta>0$,
\begin{equation}\label{eq5.14}
G_2(y)<G_1(y),\quad y\in[1-\delta,1).
\end{equation}

To get a contradiction, assume that for some $y_0\in(0,1-\delta)$ we have
\[
G_2(y_0)=G_1(y_0).
\]
Let $y_*\leq 1-\delta$ be the supremum of all such points $y_0$. Then $G_2(y_*)=G_1(y_*)$ and
\begin{equation}\label{eq5.15}
G_2(y)<G_1(y),\quad y\in(y_*,1).
\end{equation}
On the other hand, by \eqref{eq5.12} and that $G_2(y_*)=G_1(y_*)$, using $D_R$ for the right derivative,
\begin{align*}
D_R\,G_2(y_*)-D_R\,G_1(y_*)&=G_2'(y_*)-G_1'(y_*)\\
&=F_1(y_*,G_1(y_*))-F_2(y_*,G_2(y_*))\\
&=2p\left[\,\frac 1{a_1}-\frac 1{a_2}\,\right]\frac{\left[4+G_1(y_*)\right]}{y_*^2\left[4+(p-1)G_1(y_*)\right]}\\
&>0.
\end{align*}
Hence for small $\delta_1>0$,
\[
G_2>G_1\quad\text{ on }(y_*,y_*+\delta_1].
\]
This contradicts \eqref{eq5.15}. Thus there is no $y_0\in(0,1-\delta)$ for which $G_2(y_0)=G_1(y_0)$ and so by \eqref{eq5.14}, we must have $G_2<G_1$ on $(0,1)$, as claimed.

To see that $\lim_{a\to 0^+}G_a(y)=\infty$ for each $y\in(0,1)$, assume that for some $y_0\in(0,1)$,
\[
\lim_{a\to 0^+}G_a(y_0)=g(y_0)<\infty.
\]
Then since $G_a(y)$ is decreasing in both $a>0$ and $y\in(0,1)$, for any $y\in(y_0,1)$ we have that
\[
\lim_{a\to 0^+}G_a(y)\leq\lim_{a\to 0^+}G_a(y_0)<\infty.
\]
Thus $g(y)=\lim_{a\to 0^+}G_a(y)$ exists as a real number for each $y\in(y_0,1)$. By Fatou's lemma we have that
\begin{align*}
\int_{y_0}^1\liminf_{a\to 0^+}F_a(u,G_a(u))\,du&\leq \liminf_{a\to 0^+}\int_{y_0}^1F_a(u,G_a(u))\,du\\
&=\liminf_{a\to 0^+}G_a(y_0)\qquad\text{(see \eqref{eq5.13}).}
\end{align*}
On the other hand, since $G_a(u)\to g(u)$ for $u\in(y_0,1)$, by \eqref{eq5.9} we have that
\[
\lim_{a\to 0^+}F_a(u,G_a(u))=\infty.\]
Thus
\[
\infty=\liminf_{a\to 0^+}G_a(y_0)=g(y_0)<\infty;
\]
contradiction.
\end{proof}

As a consequence of Lemma \ref{lem5.2}, the function
\begin{equation}\label{eq5.16}
K(y)=\lim_{a\to \infty}G_a(y),\quad 0<y\leq 1
\end{equation}
is well-defined and $K(1)=0$.
\begin{lem}\label{lem5.3}
The function $K$ is decreasing on $(0,1]$ and is positive on $(0,1)$.
\end{lem}
\begin{proof}
Since $G_a(y)$ is strictly decreasing and nonnegative as a function of $y\in(0,1]$, $K$ must be decreasing and nonnegative on $(0,1]$. If for some $y_0\in(0,1)$ we have $K(y)=0$, then by monotonicity and nonnegativity, $K\equiv 0$ on $[y_0,1]$. By Fatou's lemma and \eqref{eq5.13},
\begin{align}\label{eq5.17}
\int_{y_0}^1F_\infty(u,K(u))\,du&=\int_{y_0}^1\lim_{a\to \infty}F_a(u,G_a(u))\,du\\
&\leq\liminf_{a\to\infty}\int_{y_0}^1F_a(u,G_a(u))\,du\notag\\
&=\liminf_{a\to\infty}G_a(y_0)\notag\\
&=K(y_0)\notag\\
&=0.\notag
\end{align}
On the other hand, for each $u\in[y_0,1]$, we have that
\begin{align*}
F_\infty(u,K(u))&=\frac{16p+4(3p-2)K(u)+2(p-1)K^2(u)}{u[\,4+(p-1)K(u)\,]}\\
&=\frac{4p}u>0,
\end{align*}
since $K\equiv 0$, on $[y_0,1]$. Thus the left hand side of \eqref{eq5.17} must be positive. This contradiction yields that we must have $K>0$ on $(0,1)$.
\end{proof}
Now we show that we can take the limit as $a\to\infty$ in \eqref{eq5.10}.
\begin{lem}\label{lem5.4}
The function $K$ is continuous on $(0,1)$ and satisfies
\begin{align*}
K'(y)&=\frac{16p+4(3p-2)K(y)+2(p-1)K^2(y)}{y[\,4+(p-1)K(y)\,]}\\
&=F_\infty(y,K(y)),\quad 0<y<1.
\end{align*}
\end{lem}
\begin{proof}
By \eqref{eq5.16}, Lemma \ref{lem5.2} and Lemma \ref{lem5.3}, for sufficiently small $\delta>0$, if $y\in(0,1-\delta)$, then
\[
G_a(y)\geq K(y)\geq K(1-\delta)>0,\quad a>0.
\]
Notice that for $u\in[\delta,1-\delta]$ and $w\geq 0$,
\begin{align*}
0&\leq F_a(u,w)-F_\infty(u,w)\\
&=\frac{2p(4+w)}{au^2[\,4+(p-1)w\,]}\\
&\leq\frac{2p}{a\delta^2}\left[1+\frac 1{p-1}\right].
\end{align*}
This implies that as $a\to\infty$, $F_a(u,w)\to F_\infty(u,w)$ uniformly for $u\in[\delta,1-\delta]$ and $w\geq 0$. In particular, for

\[
N_a=\sup\{|F_a(u,w)-F_\infty(u,w)|:\delta\leq u\leq 1-\delta,\ w\geq K(1-\delta)\}
\]
we have
\[
N_a\to 0\quad \text{as }a\to\infty.
\]
Since the $w$-partial of $F_\infty(y,w)$ is of the form
\[
\frac{f_1(y)w^2+f_2(y)w+f_3(y)}{y^2[\,4+(p-1)w\,]^2},
\]
where $f_1$--$f_3$ are bounded on $[0,1]$, we have that
\[
M:=\sup\left\{\left|\frac{\partial F_\infty}{\partial w}(y,w)\right|:\delta\leq y\leq 1-\delta,\quad w\geq K(1-\delta)\right\}<\infty.
\]
Then for any $u\in[\delta,1-\delta]$ and $w_1,w_2\geq K(1-\delta)$,
\[
\left|F_\infty(u,w_1)-F_\infty(u,w_2)\right|\leq M|w_1-w_2|,
\]
and consequently,
\begin{align*}
\left|F_{a_1}(u,w_1)-F_{a_2}(u,w_2)\right|&\leq\left|F_{a_1}(u,w_1)-F_{\infty}(u,w_1)\right|\\
&\qquad +\left|F_{\infty}(u,w_1)-F_{\infty}(u,w_2)\right|+\left|F_{\infty}(u,w_2)-F_{a_2}(u,w_2)\right|\\
&\leq N_{a_1}+M|w_1-w_2|+N_{a_2}.
\end{align*}
Hence for any $a_1,a_2>1$ and $y\in[\delta,1-\delta]$, by \eqref{eq5.13} we have that
\begin{align*}
\left|G_{a_1}(y)-G_{a_2}(y)\right|&=\left|G_{a_1}(\delta)-\int_\delta^yF_{a_1}(u,G_{a_1}(u))\,du-G_{a_2}(\delta)+\int_\delta^yF_{a_2}(u,G_{a_2}(u))\,du\right|\\
&\leq \left|G_{a_1}(\delta)-G_{a_2}(\delta)\right|+\int_\delta^{y}\left|F_{a_1}(u,G_{a_1}(u))-F_{a_2}(u,G_{a_2}(u))\right|\,du\\
&\leq \left|G_{a_1}(\delta)-G_{a_2}(\delta)\right|+N_{a_1}+M\int_\delta^{y}\left|G_{a_1}(u)-G_{a_2}(u)\right|\,du+N_{a_2}.
\end{align*}
By Gronwall's inequality,
\[
\sup_{\delta\leq y\leq 1-\delta}\left|G_{a_1}(y)-G_{a_2}(y)\right|\leq \left[\left|G_{a_1}(\delta)-G_{a_2}(\delta)\right|+N_{a_1}+N_{a_2}\right]e^{M(1-2\delta)}.
\]
Since $G_a(\delta)\to K(\delta)$ as $a\to\infty$, it follows that for any sequence $a_n\to\infty$, $G_{a_n}$ is uniformly Cauchy on $[\delta,1-\delta]$. Since $\delta>0$ was arbitrary, we get that $K$ is continuous on $(0,1)$.

Then by the uniform convergence of $G_a$ to $K$ on compact subsets of $(0,1)$ as $a\to\infty$, we can take the limit inside the integral in the expression
\[
G_a(y)=G_a(\delta)-\int_\delta^yF_a(u,G_a(u))\,du,\quad\delta\leq y\leq 1-\delta
\]
to get
\[
K(y)=K(\delta)-\int_\delta^yF_\infty(u,K(u))\,du,\quad\delta\leq y\leq 1-\delta.
\]
Hence, since $\delta>0$ was arbitrary,
\[
K'(y)=-F_\infty(y,K(y)),\quad 0<y<1.
\]
\end{proof}

\begin{lem}\label{lem5.5}
The function $K$ is continuous on $(0,1]$.
\end{lem}
\begin{proof}
By Lemma \ref{lem5.4}, it suffices to show that $\lim_{y\to 1^-}K(y)=0$. Given $\vp>0$, choose $\delta>0$ such that
\[
G_1(y)<\vp,\quad y\in[1-\delta,1].
\]
Then
\[
0\leq G_a(y)<\vp\quad\text{for all $y\in[1-\delta,1]$ and $a\geq 1$}.
\]
Let $a\to\infty$ to get that
\[
0\leq K(y)\leq\vp,\quad y\in[1-\delta,1].
\]
Thus $K(y)\to 0$ as $y\to 1^-$, as desired.
\end{proof}
To prove integrability properties of $K$ on $(0,1)$, we will make use of the following result.
\begin{lem}\label{lem5.6}
We have the following limits:

\bigskip
a) $\displaystyle\lim_{y\to 1^-}\frac{K(y)}{1-y}=4p$.

\bigskip
b) $\displaystyle\lim_{y\to 0^+}\frac{\log K(y)}{-\log y}=2$.

\end{lem}
\begin{proof}
a) Since $K(y)\to 0$ as $y\to 1^-$, we have 
\begin{align*}
\lim_{y\to 1^-}\frac{K(y)}{1-y}&=\lim_{y\to 1^-}\frac{K'(y)}{-1}\\
&=\lim_{y\to 1^-}F_\infty(y,K(y))\\
&=4p.
\end{align*}

\noindent b) Now for $y>0$ and $w\geq 0$,
\[
F_\infty(y,w)\geq\frac{16p}{y\left[4+(p-1)w\right]}.
\]
To get a contradiction, assume
\[
\lim_{y\to 0^+}K(y)=L<\infty.
\]
Then since $K$ is decreasing on $(0,1)$, we have $K(u)\leq L$ for $u\in(0,1)$ and consequently for $y>0$,
\begin{align*}
K(y)&=\int_y^1F_\infty(u,K(u))\,du\\
&\geq\int_y^1\frac{16p}{u\left[4+(p-1)L\right]}\,du\\
&=-\frac{16p}{4+(p-1)L}\log y.
\end{align*}
Letting $y\to 0^+$, this yields $\infty>L=\lim_{y\to 0^+}K(y)=\infty$; contradiction.

Thus $K(y)\to\infty$ as $y\to 0^+$ and we have that
\begin{align*}
\lim_{y\to 0^+}\frac{\log K(y)}{-\log y}&=\lim_{y\to 0^+}\frac{K'(y)/K(y)}{-1/y}\\
&=\lim_{y\to 0^+}\frac{y}{K(y)}F_\infty(y,K(y))\\
&=\lim_{y\to 0^+}\frac 1{K(y)}\frac{16p+4(3p-2)K(y)+2(p-1)K(y)^2}{4+(p-1)K(y)}\\
&=2.
\end{align*}
\end{proof}
As an immediate consequence of the lemma, we get the following corollary.
\begin{cor}\label{cor5.7}
The function $y^{-1}K(y)^{-1/2}$ is integrable on $(0,1)$.\hfill $\square$
\end{cor}
For each $a>0$, define
\begin{equation}\label{eq5.18}
H_a(y)=y^2G_a(y/a),\quad 0<y\leq a.
\end{equation}
Then $H_a$ is continuous on $(0,a]$, $H_a>0$ on $(0,a)$ and $H_a(a)=0$. Moreover, $H_a$ satisfies
\begin{equation}\label{eq5.19}
H_a'(y)=-\frac{8py^2[\,1+2y\,]+[\,2p+4(3p-4)y\,]H_a(y)}{4y^2+(p-1)H_a(y)},\quad 0<y<a.
\end{equation}
\begin{lem}\label{lem5.8}
The function $H_a$ is positive on $[0,a)$ and in $C^\infty([0,a])$.
\end{lem}
\begin{proof}
By Lemma \ref{lem5.1}, $G_a\in C^\infty((0,1])$, so we have $H_a\in C^\infty((0,a])$. 

Now if $y>0$ is very small, $2p+4(3p-4)y>0$. Hence $H_a'(y)<0$ for small $y>0$, and we have that $H_a(y)$ is increasing as $y$ decreases to $0$. Thus
\[
L:=\lim_{y\to 0^+}H_a(y)
\]
exists as an extended real number in $(0,\infty]$.

If $L=\infty$, then
\[
\lim_{y\to 0^+}H'_a(y)=-\frac{2p}{p-1}.
\]
Since
\[
H_a(y)-\int_\vp^yH'_a(u)\,du=H_a(\vp)
\]
for $0<\vp<y$, we get a contradiction when $\vp\to 0^+$; thus we must have $L<\infty$. Then by an extension theorem in Coddington and Levinson ((1955), Theorem 1.3 on page 47), for some $\delta>0$ the function $H_a(y)$ can be uniquely extended to be the solution in $C^1((-\delta,a))$ of \eqref{eq5.19} satisfying the initial condition $H_a(a)=0$. Moreover, by making $\delta$ smaller if necessary, $H_a>0$ on $(-\delta,a)$ and so by repeatedly differentiating \eqref{eq5.19}, we get that $H_a\in C^\infty((-\delta,a))$.
\end{proof}

Since $G_a\geq K$ on $(0,1)$, by Corollary \ref{cor5.7}, $(H_a(y))^{-1/2}=y^{-1}(G_a(y/a))^{-1/2}$ is integrable on $(0,a)$. Thus the function
\begin{equation}\label{eq5.20}
g(t)=\int_t^a\frac{dy}{\sqrt{H_a(y)}},\quad t\in[0,a]
\end{equation}
is well-defined and continuous on $[\,0,a\,]$. Furthermore, since $g$ is strictly decreasing on $[\,0,a\,]$, it has a continuous and strictly decreasing inverse that we denote by $y(\theta)$. Thus setting
\begin{equation}\label{eq5.21}
\theta_a=g(0),
\end{equation}
we have that the domain of $y(\theta)$ is $[\,0,\theta_a\,]$. In particular,
\begin{equation}\label{eq5.22}
\theta=\int_{y(\theta)}^a\frac{dw}{\sqrt{H_a(w)}},\quad \theta\in[\,0,\theta_a\,].
\end{equation}
From this we get
\begin{equation}\label{eq5.23}
 \left\{\begin{array}{l}
y'(\theta)=-\sqrt{H_a(y(\theta))},\quad 0<\theta<\theta_a\\
\noalign{\smallskip}
y(\theta_a)=0,\\
\noalign{\smallskip}
y(0)=a\\
\noalign{\smallskip}
y>0\text{ on }[\,0,\theta_a).
        \end{array}\right.
\end{equation}
Now extend $y(\theta)$ to $[\,-\theta_a,\theta_a\,]$ by
\[
y(\theta)=y(-\theta),\quad\theta\in[\,-\theta_a,0).
\]
\begin{lem}\label{lem5.9}
The extended function $y(\theta)$ is in $C^3([-\theta_a,\theta_a])$ and satisfies the equation
\[
4y^2\left[\,1+2y+\tfrac 1py''\,\right]+(y')^2\left[\,1+\tfrac{2(3p-4)}p\,y+\tfrac{p-1}p\,y''\,\right]=0
\]
on $(-\theta_a,\theta_a)$. Moreover, $y>0$ on $(-\theta_a,\theta_a)$ and $y(\pm\theta_a)=0$.
\end{lem}
\begin{proof} By \eqref{eq5.23}, $y'=-\sqrt{H_a(y)}$ on $\left(0,\theta_a\right)$. Since $H_a$ satisfies \eqref{eq5.19}, it is a routine matter to check that $y$ satisfies the indicated differential equation on $\left(0,\theta_a\right)$. Then using that $y'(\theta)=-y'(-\theta)$ for $\theta\in\left(-\theta_a,0\right)$, it follows that $y$ also solves the differential equation on $\left(-\theta_a,0\right)$.

By \eqref{eq5.23}, $y>0$ on $(-\theta_a,\theta_a)$ with $y(\pm\theta_a)=0$. Since $y((0,\theta_a])=[0,a)$ and $H_a>0$ on $[0,a)$, by Lemma \ref{lem5.8} we can repeatedly differentiate the differential equation in \eqref{eq5.23} to see that $y\in C^\infty((0,\theta_a])$. Thus $y\in C^\infty([-\theta_a,0)\cup (0,\theta_a])$. 

All that remains is to show that $y$ is three times continuously differentiable at $0$. We make use of the following fact: Suppose $f$ is continuous on $(-\delta,\delta)$ and differentiable on $(-\delta,\delta)\backslash\{0\}$. If $\lim_{h\to 0^-}f'(h)=L=\lim_{h\to 0^+}f'(h)$, then the left and right derivatives of $f$ at $0$ exist and are equal to $L$. Thus $f$ is continuously differentiable at $0$ and the value of the derivative there is $L$.

Now on $(0,\theta_a)$,
\begin{align*}
&y'=-\sqrt{H_a(y)},\\
&y''=\frac 12H'_a(y),\\
&y^{(3)}=\frac 12H_a''(y)\,y'.
\end{align*}
Then on $(-\theta_a,0)$,
\begin{align*}
&y'(\theta)=-y'(-\theta),\\
&y''(\theta)=y''(-\theta),\\
&y^{(3)}(\theta)=-y^{(3)}(-\theta),
\end{align*}
and we have
\begin{align*}
\lim_{\theta\to 0^+}y'(\theta)&=-\lim_{\theta\to 0^+}\sqrt{H_a(y(\theta))}\\
&=-\sqrt{H_a(a)}\\
&=0.
\end{align*}
Thus
\[
\lim_{\theta\to 0^-}y'(\theta)=-\lim_{\theta\to 0^-}y'(-\theta)=-\lim_{\theta\to 0^+}y'(\theta)=0
\]
and so $y$ is continuously differentiable at $0$ with $y'(0)=0$.

Next,
\begin{align*}
\lim_{\theta\to 0^+}y''(\theta)&=\frac 12\,\lim_{\theta\to 0^+}H_a'(y(\theta))\\
&=-p(1+2a),\quad\text{by \eqref{eq5.19},}
\end{align*}
and so
\[
\lim_{\theta\to 0^-}y''(\theta)=\lim_{\theta\to 0^-}y''(-\theta)=\lim_{\theta\to 0^+}y''(\theta)=-p(1+2a).
\]
Thus $y'$ is continuously differentiable at $0$ and $y''(0)=-p(1-2a)$.

Using \eqref{eq5.19}, it is easy to check that $\lim_{\theta\to 0^+}H''(y(\theta))$ exists. Hence we have that
\begin{align*}
\lim_{\theta\to 0^+}y^{(3)}(\theta)&=\frac 12\,\lim_{\theta\to 0^+}H''_a(y(\theta))\,y'(\theta)\\
&=0,
\end{align*}
and consequently,
\[
\lim_{\theta\to 0^-}y^{(3)}(\theta)=-\lim_{\theta\to 0^-}y^{(3)}(-\theta)=-\lim_{\theta\to 0^+}y^{(3)}(\theta)=0.
\]
It follows that $y''$ is continuously differentiable at $0$ and $y^{(3)}(0)=0$.
\end{proof}

\begin{lem}\label{lem5.10}
The root $\theta_a$ is continuous and increasing as a function of $a>0$. Moreover,
\[
\lim_{a\to\infty}\theta_a=\frac\pi 2\left[\,1-\frac 12\sqrt{\frac{2(p-1)}p}\,\right]:=:\eta_p
\]
and the range of $\theta_a$ as a function of $a>0$ is $(0,\eta_p)$.
\end{lem}
\begin{proof}
Observe that
\begin{align}\label{eq5.24}
\theta_a&=\int_0^a\frac{dw}{\sqrt{H_a(w)}}\notag\\
&=\int_0^a\frac{dw}{w\sqrt{G_a(w/a)}}\notag\\
&=\int_0^1\frac{du}{u\sqrt{G_a(u)}}.
\end{align}
Since $G_a(\,\cdot\,)$ is decreasing as a function of $a>0$, $\theta_a$ is increasing for $a>0$. By Lemma \ref{lem5.1}, since $G_a\geq K$ and $y^{-1}(K(y))^{-1/2}$ is integrable, we get that $\theta_a$ is a continuous function of $a>0$. Then by monotonicity,
\begin{align*}
\lim_{a\to\infty}\theta_a&=\lim_{a\to\infty}\int_0^1\frac{du}{u\sqrt{G_a(u)}}\\
&=\int_0^1\frac{du}{u\sqrt{K(u)}}.
\end{align*}
Upon changing variables $y=K(u)$ and using Lemma \ref{lem5.6}b together with the expression for $K'(u)$ from Lemma \ref{lem5.4}, we get that
\[
\lim_{a\to\infty}\theta_a=\int_0^\infty\frac{4+(p-1)y}{y^{1/2}\left[\,16p+4(3p-2)\,y+2(p-1)\,y^2\,\right]}\,dy.
\]
After another change of variables $y=z^2$, this becomes
\[
\lim_{a\to\infty}\theta_a=2\int_0^\infty\frac{4+(p-1)z^2}{2(p-1)z^4+4(3p-2)z^2+16p}\,dz.
\]
The latter integral is easily evaluated using residue theory to yield
\[
\lim_{a\to\infty}\theta_a=\frac\pi 2\left[\,1-\frac 12\sqrt{\frac{2(p-1)}p}\,\right],
\]
as desired.

To see that the range of $\theta_a$ is $(0,\eta_p)$, it is enough to show that
\[
\lim_{a\to 0^+}\theta_a=0.
\]
Since $u^{-1}(G_a(u))^{-1/2}\leq u^{-1}(K(u))^{-1/2}$ and the latter is integrable, by Lemma \ref{lem5.2} and dominated convergence, \eqref{eq5.24} yields
\[
\lim_{a\to 0^+}\theta_a=\lim_{a\to 0^+}\int_0^1\frac{du}{u\sqrt{G_a(u)}}=0.
\]
\end{proof}

At last we can prove Theorem \ref{thm1.2}. Let 
\[
\eta<\pi\left[\,1-\frac 12\sqrt{\frac{2(p-1)}p}\,\right].
\]
By Lemma \ref{lem5.10}, choose $a>0$ such that $\theta_a=\frac\eta 2$. Taking $f(\theta)$ to be the corresponding $y(\theta)$ from Lemma \ref{lem5.9} and setting $u=r^2f(\theta)$ does the trick.


\newpage
\begin{thebibliography}{13}

\bibitem{AP} 
B. Abdellaoui and I. Peral (2003). Existence and nonexistence results for quasilinear elliptic equations involving the p-Laplacian with a critical potential, Ann. Mat. Pura Appl. {\bf 182} 247--270.

\bibitem{Ar} 
G. Aronsson (1986). Construction of singular solutions to the $p$-harmonic equation and its limit equation for $p=\infty$, Manuscripta Math. {\bf 56} 135--158.

\bibitem{BDS} 
R. Ba\~nuelos, R. D. DeBlassie and R. G. Smits (2001). The first exit time of planar Brownian motion from the interior of a parabola, Annals of Probability {\bf 29} 882--901.

\bibitem{BS} 
R. Ba\~nuelos and R. G. Smits (1997). Brownian motion in cones, Probab. Theory Relat. Fields {\bf 108} 299-319.

\bibitem{B-V} 
M.-F. Bidaut-V\'eron (1989). Local and global behaviors of solutions of quasilinear equations of Emden-Fowler type, Arch. Rational Mech. Anal. {\bf 107} 293--324.

\bibitem{BR} 
B.M. Brown and W. Reichel (2004). Eigenvalues of the radially symmetric p-Laplacian in $\bb{R}^n$, J. London Math. Soc. {\bf 69} 657--675.

\bibitem{Bu1} 
D.L. Burkholder (1973). Distribution function inequalities for martingales, Annals of Probability {\bf 1} 19--42.

\bibitem{Bu2} 
D.L. Burkholder (1977). Exit times of Brownian motion, harmonic majorization and Hardy spaces, Adv.\ Math. {\bf 26} 182--205.

\bibitem{CL} 
E.A. Coddington and N. Levinson (1955). \emph{Theory of Ordinary Differential Equations}, McGraw Hill, New York.

\bibitem{DZ} 
B. Davis and B. Zhang (1994). Moments of the lifetime of conditioned Brownian motion in cones, Proceedings of the American Mathematical Society {\bf 121} 925--929.

\bibitem{De} 
R. D. DeBlassie (1987). Exit times from cones in $\bb{R}^n$ of Brownian motion, Probab. Theory Relat. Fields {\bf 74} 1--29.

\bibitem{D}
P. Dr\'{a}bek (1995). Nonlinear Eigenvalue problems for the p-Laplacian in $\bb{R}^n$, Math. Nachr. {\bf 173} 131--139.

\bibitem{DY} 
E.B. Dynkin and A.A. Yushkevich (1969), \emph{Markov Processes: Theorems and Problems}, Plenum, New York.

\bibitem{FMST}
J. Fleckinger, R. F. Man\'asevich, N. M. Stavrakakis and F. de Th\'elin (1997). Principal eigenvalues for some quasilinear elliptic equation in $\bb{R}^n$, Adv. Differential Equations {\bf 2} 981--1003.

\bibitem{FHT}
J. Fleckinger, E. M. Harrell II, N. M. Stavrakakis and F. de Th\'elin (1999). Boundary behavior and estimates for solutions of equations containing the p-Laplacian, Electronic Journal of Differential Equations {\bf 1999} No. 38, 1--19.

\bibitem{Hu} 
G. A. Hunt (1956). Some theorems concerning Brownian motion, Transactions of the American Mathematical Society {\bf 81} 294--319.

\bibitem{Ka} 
S. Kakutani (1944). Two-dimensional Brownian motion and harmonic functions, Proc. Imp. Acad. Tokyo {\bf 20} 706--714

\bibitem{Kr}
A. Krist\'aly (2004). An existence result for gradient-type systems with a nondifferentiable term on unbounded strips, J. Math. Anal. Appl. {\bf 299} 186--204.

\bibitem{LLPU} 
A.J. Lazarus, D.E. Loeb, J.G. Propp and D. Ullman (1996). Richman games.  \emph{Games of no chance} (Berkeley, CA, 1994),  439--449, Math. Sci. Res. Inst. Publ., 29, Cambridge Univ. Press, Cambridge.

\bibitem{LLM}
V. Liskevich, S. Lyakhova and V. Moroz (2007). Positive solutions to nonlinear p-Laplace equations with Hardy potential in exterior domains, J. Differential Equations {\bf 232} 212--252.

\bibitem{MS} 
A. Maitra and W. Sudderth (1998). Finitely additive stochastic games with Borel measurable payoffs, International Journal of Game Theory {\bf 27} 257--267.

\bibitem{MR}
E. Montefusco and V. R\u{a}dulescu (2001). Nonlinear eigenvalue problems for quasilinear operators on unbounded domains, Nonlinear Differential Equations and Applications {\bf 8} 481--497.

\bibitem{Pa} 
D.E. Panayotounakos (2005). Exact analytic solutions of unsolvable classes of first and second order nonlinear ODEs (Part I: Abel's equations), Applied Mathematics Letters {\bf 18} 155--162.

\bibitem{PS} 
Y. Peres and S. Sheffield (2008). Tug-of-war with noise: a game-theoretic view of the $p$-Laplacian, Duke Mathematical Journal {\bf 145} 91--120.

\bibitem{PSSW1} 
Y. Peres, O. Schramm, S. Sheffield and D.B. Wilson (2007). Random-turn hex and other selection games, Amer. Math. Monthly {\bf 114} 373--387.

\bibitem{PSSW2} 
Y. Peres, O. Schramm, S. Sheffield and D.B. Wilson (2009). Tug-of-war and the infinity Laplacian, J. Amer. Math. Soc. {\bf 22} 167--210.

\bibitem{Pf}
K. Pfl\"uger (1998). Existence and multiplicity of solutions to a p-Laplacian equation with nonlinear boundary condition, Electronic Journal of Differential Equations {\bf 1998} No. 10, 1--13.

\bibitem{PZ} 
A.D. Polyanin and V.F. Zaitsev (2003). \emph{Handbook of Exact Solutions for Ordinary Differential Equations, $2^\text{nd}$ edition}, Chapman \& Hall/CRC, Boca Raton.

\bibitem{Sp} 
J. Spencer (1977). Balancing goods, J. Combinatorial Theory Ser. B {\bf 23} 68--74.

\bibitem{Yu}
L.-S. Yu (1992). Nonlinear p-Laplacian problems on unbounded domains, Proceedings of the American Mathematical Society {\bf 115} 1037--1045.

\end{thebibliography}
\end{document}